\documentclass[a4paper,12pt]{article}

\usepackage[latin1]{inputenc}
\usepackage[english]{babel}
\usepackage{amssymb}
\usepackage{amsmath}
\usepackage{latexsym}
\usepackage{amsthm}
\usepackage[pdftex]{graphicx}
\usepackage{tikz} %questo package va messo dopo graphicx
\usepackage{hyperref}
\usepackage{pgfplots}  
\usepackage{mathtools}
\pgfplotsset{width=6.6cm,compat=1.7}

\usetikzlibrary{cd}

\usetikzlibrary{decorations.pathreplacing}

\DeclareMathOperator{\id}{id}
\DeclareMathOperator{\Grid}{Grid}

 \theoremstyle{plain}
 \newtheorem{thm}{Theorem}[section]
 \newtheorem{cor}[thm]{Corollary}
 \newtheorem{lem}[thm]{Lemma}
 \newtheorem{prop}[thm]{Proposition}
 \newtheorem{question}[thm]{Question}

 \theoremstyle{definition}
 
 \newtheorem{example}[thm]{Example}

 \theoremstyle{remark}
 
\title{Groups generated by pattern avoiding permutations}
\date{}
\author{Marilena Barnabei\thanks{corresponding author} \\
P.A.M. \\
Universit\`a di Bologna, 40126, ITALY \\
\texttt{marilena.barnabei@unibo.it}\and
Niccol\`o Castronuovo \\
Liceo ``A. Einstein'', Rimini, 47923, ITALY \\
\texttt{castronuovoniccolo@gmail.com}\and
Matteo Silimbani \\
Istituto Comprensivo ``E. Rosetti'', Forlimpopoli, 47034, ITALY \\
\texttt{matteosilimbani@icrosetti.istruzioneer.it}}

\begin{document}
\maketitle
\begin{abstract}
We study groups generated by sets of pattern avoiding permutations. In the first part of the paper we prove some general results concerning the structure of such groups. In the second part we carry out a case-by-case analysis of groups generated by permutations avoiding few short patterns. 
\end{abstract}



\section{Introduction}

The symmetric group $S_n$ over $n$ elements is ubiquitous in mathematics, indeed, it is studied both from a group-theoretical  and a combinatorial point of view. 

In the first case the elements of  $S_n$ are thought as bijections  acting on the set $\{1,\,2,\,\ldots\,, n\}$ and are written as sequences of cycles, in the second case they are considered as words containing all the symbols from $1$ to $n$ once and are written in one-line notation as the sequence of such symbols.

The group-theoretical study of the symmetric group embodies the wide fields of  representation theory of the symmetric group (see, e.g., \cite{sagansymmetric}) and permutation groups (see, e.g., \cite{cameron1999permutation}).
The combinatorics of permutations (see, e.g., \cite{bona2022combinatorics}) includes many topics among whom one of the most recently developed is the study of pattern avoiding permutations (see, e.g., \cite{Ki}).

The two aspects are rarely considered together, due to the different nature of the two approaches. Papers that deal with both the perspectives are, for instance,  
\cite{Albert2007CompositionsOP,Atkinson2001PermutationIA,BurcroffPatternavoidingPP,KarpilovskijComposabilityOP, LehtonenPermutationGA,LehtonenPermutationGP}.

In this paper we consider the problem of determining the group generated by a set of pattern avoiding permutations.  This topic has never been investigated, to the best of our knowledge.

We begin with introducing some notations and presenting some preliminary results (Sections \ref{notations_defs} and \ref{section_lemmas}).

Then, given a set of patterns $T,$ we consider the subset $S_n(T)$ of the symmetric group $S_n$ consisting of permutations avoiding each pattern in $T,$ and turn our attention to the group $G_n$ generated by $S_n(T).$

When dealing with such groups some questions naturally arise. First of all, under which conditions the sequence $(G_n)_{n\geq 0}$ is eventually constant and which kind of groups arise?

In Section \ref{section_Sk} we find a set of patterns such that $(G_n)_{n\geq 0}$ is eventually constant and equal to an assigned symmetric group $S_k.$ In Section \ref{section_alternating} we show that the alternating group can not be obtained in this way and that the only non-trivial abelian groups that arise in this way, are $\mathbb Z_2$ and $\mathbb Z_2\times \mathbb Z_2.$
The proof of such results take advantage of a characterization of those sets of patterns such that the sequence $(G_n)$ has polynomial growth.

A further question is under which conditions $G_n=S_n$ for every sufficiently large $n.$ 

In Section \ref{section_gen_Sn} we single out a fairly general condition under which this happens. Sections \ref{section_T_three} and \ref{section_T_four} are devoted to the analysis of the groups generated by $S_n(T)$ when $T$ consists of few patterns or  patterns of small length. In the last section we present some open problems.

\section{Notations and definitions}\label{notations_defs}

Firstly we recall some notations from group theory.
We denote by
\begin{itemize}
\item  $S_n$ the symmetric group, i.e., the set of permutations of  $[n]=\{1,2,\ldots ,n\},$ with the operation of composition, and by $S$ the union of all these sets, $$S=\bigcup_{n\geq 0} S_n,$$

\item  $A_n$ the alternating group, i.e., the subgroup of $S_n$ of even permutations,
\item  $\mathbb Z_n$ the cyclic group with $n$ elements,
\item  $D_n$ the dihedral group with $2n$ elements, i.e., the group of symmetries of the regular $n$-agon.
\end{itemize}

Recall that every element $\pi$ of $S_n$ can be written either as a product of cycles, or in one-line notation $\pi=\pi_1\,\pi_2\,\ldots\,\pi_n$ as the sequence of images of the elements $1,2,\ldots,n.$ The presence or not of the brackets $(,\,)$ will make clear which  notation we are using. We write products from right to left, that is, if $\pi,\sigma \in S_n,$ $\pi \sigma$ is the permutation obtained by applying  $\sigma$ first.

Given a set $A \subseteq S_n,$ the symbol $\langle A\rangle$ will denote the group generated by $A.$ 

We will use the following notation for special elements of $S_n.$
\begin{itemize}
\item the identity permutation will be denoted by $\id_n,$ $$\id _n=1\,2\,\ldots \,n,$$
\item the decreasing permutation will be denoted by $\psi_n,$ $$\psi_n= n\quad n-1\,\ldots\,2\;  1=\prod_{i=1}^{\lfloor\frac{n}{2}\rfloor} (i,\; n+1-i).$$
\end{itemize}
We will use the symbols $\id$ and $\psi$ when the size is clear from the context. 

If $\pi\in S_n$ we will say that $\pi$ has \textit{length} $n$ and write $|\pi|=n.$ A permutation $\pi$ such that $\pi=\pi^{-1}$ is called an \textit{involution}. 

The \textit{reverse}, the \textit{complement} and \textit{reverse-complement} of the permutation $\pi$
are the permutations $\pi^r=\pi \psi,$ $\pi^c=\psi \pi$ and $\psi^{rc}=\psi \pi\psi,$ respectively.  

A permutation $\pi\in S_n$ \textit{contains the pattern} $\tau\in S_k$ if there are indices $i_1,i_2,\ldots,i_k$ such that the subsequence $\pi_{i_1}\,\pi_{i_2}\,\ldots\,\pi_{i_k}$
is order isomorphic to $\tau,$ and we write $\pi\geq \tau.$

Otherwise we say that $\pi$ \textit{avoids the pattern}  $\tau.$ 

The set of permutations avoiding all the patterns in the set $T=\{\tau_1,\tau_2,\ldots\}$ is denoted by $S_n(T)$ or, with a slight abuse of notation, by $S_n(\tau_1,\tau_2,\ldots).$

We set $S(T)=\cup_{n\geq 0} S_n(T).$

Notice that the set $S$ with the containment relation $\leq$ is a poset. A downset is this order is called a \textit{permutation class} or simply a \textit{class}. Given a class $\mathcal{C},$ let $\mathcal{C}_n=\mathcal{C}\cap S_n$ for every $n\geq 0.$
A class can be specified by its \textit{basis}, the minimal permutations not in the class.
If $T$ is the basis of a class $\mathcal{C},$ then $\mathcal{C}=S(T).$
Observe that the basis of a class can be infinite, since the poset $(S,\leq)$ contains infinite antichains. 
% If the basis of a class $\mathcal{C}$ is finite $\{\tau_1,\tau_2,\ldots,\tau_k\}$ than $\mathcal{C}_n=S_n(\tau_1,\tau_2,\ldots,\tau_k).$

\section{Preliminary results}\label{section_lemmas}

Now we state some basic Lemmas that will be useful in the following.
\begin{lem}
\label{inclusion}
Let $T$ and $T'$ be two subsets of $S$, 
 with $T\subseteq T'$. Then $\langle S_n(T')\rangle \subseteq \langle S_n(T)\rangle$.
\end{lem}
\proof
Immediate consequence of the definition of pattern containment.  
\endproof

\begin{lem}\label{sottopatterns}
Let $\sigma_1,\ldots,\sigma_k$ and $\tau_1,\ldots,\tau_k$ be permutations of any length such that $\sigma_i\leq \tau_i$ for every $i$. Then
$\langle S_n(\sigma_1,\ldots,\sigma_k) \rangle \subseteq \langle S_n(\tau_1,\ldots,\tau_k) \rangle.$
\end{lem}

\begin{proof}
This fact is a straightforward consequence of the following inclusion:
$S_n(\sigma_1,\ldots,\sigma_k) \subseteq  S_n(\tau_1,\ldots,\tau_k).$
\end{proof}

Given a set of permutations $A,$ let $A^r$ be the set $\{ \sigma^r\; |\; \sigma \in A\}.$ The sets $A^c,$ $A^{rc},$ $A^{-1}$ are defined similarly.
Notice that  $S_n(T)^r=S_n(T^r),$ and the same is true for the complement, the reverse-complement and the inverse map. Moreover, since $(\pi^r)^{-1}=(\pi^{-1})^c,$ $$(S_n(T^r))^{-1}=S_n((T^r)^{-1})=S_n((T^{-1})^c).$$
A similar identity holds with $r$ and $c$ interchanged. 

\begin{lem}\label{rc_inv}
Let $T\subseteq S$ be a set of permutations of any length. Then
$$\langle S_n(T) \rangle = \langle S_n(T^{-1}) \rangle=\langle S_n(T)\cup S_n(T^{-1}) \rangle,$$
 and $$\langle S_n(T) \rangle \cong \langle S_n(T^{rc}) \rangle.$$
\end{lem}

\begin{proof}
Given $\pi \in S_n,$ $\pi^{rc}=\psi \pi  \psi.$ 
This fact implies the second assertion. The first assertion is trivial.
\end{proof}

\begin{lem}\label{rcifnotpsi}
Let $T\subseteq S$ be a set of permutations of any length. If $\psi_k \notin T$ for every $k\geq 1,$ then $$\langle S_n(T)\rangle =\langle S_n(T)\cup S_n(T^r)\cup S_n(T^c)\cup S_n(T^{rc})\rangle.$$
\end{lem}

\proof

Notice that if $\psi_k \notin T$ for every $k\geq 1,$ then $\psi_n \in S_n(T).$  

Since $\sigma^r=\sigma  \psi,$ $\sigma^c=\psi \sigma$ and $\sigma^{rc}=\psi \sigma  \psi,$ the assertion follows.

\endproof

% \todo{ATTENZIONE  questo NON implica che $\langle S_n(T)\rangle =\langle S_n(T^r)\rangle$ perch\'e non \'e detto che nel secondo insieme ci sia $\psi$ con cui giocare}

% \todo{sono in dubbio se nella precedente sfilza di generatori aggiungere anche $T^{-1},$ $(T^{r})^{-1}$ ecc}

We recall a well-known result that presents some families of generating sets for the group $S_n.$

\begin{lem}\label{gen_sets}
The following sets generate $S_n,$ for every $n\geq 1.$
\begin{itemize}
    \item $\{(j,\;j+1),\;1\leq j\leq n-1\}.$
    \item $\{(1,\;j),\;2\leq j\leq n\}.$ 
    \item $\{(a,\;b),\;(1,\;2,\ldots ,n)\}$ if $b-a$ and $n$ are coprime.
    
\end{itemize}
\end{lem}

% \proof
% The first and second of these generating sets for the symmetric group are well known.

% The proof that also the third one is a  generating set 
%  can be found, e.g., in \cite{Conrad_gen}.

% \endproof

In the following we will use the notion 
of \textit{direct product} and \textit{semidirect product} of groups (we refer to \cite{isaacs2008finite} for definitions and notations).
In particular we will exploit the following result stated and proved in \cite[p. 69]{isaacs2008finite}.
\begin{lem}\label{sdp}
If $G$ is a group, $G$ is the semidirect product of its subgroups $N$ and $H$ if and only if the following three conditions are verified
\begin{itemize}
    \item $N\unlhd G,$
    \item $G=NH,$
    \item $N\cap H=\{ id\}.$
\end{itemize}
\end{lem}

Finally we recall a lemma which expresses a connection between the notion of pattern containment and of permutation group. See \cite{LehtonenPermutationGP} for the proof.  

\begin{lem}\label{group_patterns}
Let $G$ be any subgroup of a symmetric group $S_k.$
Then $S_n(S_k\setminus G)$ is a subgroup of $S_n.$
\end{lem}

\section{Generating a fixed symmetric group $S_k.$}\label{section_Sk}

First of all we ask the following question. 

\begin{question}\label{first_q}
Let $G$ be any finite group. Is it possible to find a subset $T$ of $S$ and a positive integer $n$ such that $\langle S_n(T) \rangle=G?$ 
\end{question}

Notice that if $T$ can depend on $n,$ then the answer to the previous question is trivially true for every $G$. In fact, if $G$ is a finite group, by Cayley's theorem there exists $n$ such that there is an isomorphic copy $\widehat G$ of $G$ into $S_n.$ Then $$G\cong \langle S_n(S_n\setminus \widehat G)\rangle=S_n(S_n\setminus \widehat G).$$

If the set of patterns $T$ must be independent of $n$ the question is more subtle.

The next theorem  shows that the question can be answered affirmatively for $G=S_k,$ for any $k.$ 

\begin{thm}\label{generatingSk}
We have
$$\langle S_n(132,231,321,k\,1\,2\,\ldots\,k-2\quad k-1) \rangle \cong S_{k-1}\quad \forall n\geq k.$$
\proof
The set $S_n(132,231,321,k\,1\,2\,\ldots\,k-2\quad k-1)$ contains precisely the permutations 
$\id$ and $t\,1\,2\,3\,\ldots\,t-1\quad t+1\,\ldots \,n-1\quad n$
for $2\leq t\leq k-1.$
In fact, if $\pi$ is in $S_n(132,231,321,k\,1\,2\,\ldots\,k-2\quad k-1),$ the first symbol $t=\pi_1$ must be smaller or equal to $k-1$ in order to avoid the patterns $321$ and $k\,1\,2\,\ldots\,k-2\quad k-1.$ If $t=1,$ $\pi=\id$ because the other symbols must be in increasing order, since $\pi$ avoids $132.$
If $t\geq 2,$ the second symbol $\pi_2$ must be equal to 1 in order to avoid $231$ and $321.$ All the other symbols must be in increasing order, in order to avoid $132.$

In particular $S_n(132,231,321,k\,1\,2\,\ldots\,k-2\quad k-1)$ consists of the identity and the cycles $(1,\,2),$ $(1,\,3,\,2),$ $(1,\,4,\,3,\,2),$ $\ldots$ $(1\,k-1\quad k-2\,\ldots \,2).$ Hence, it generates a group isomorphic to $S_{k-1}.$

\endproof

\end{thm}

\section{Alternating groups and abelian groups}\label{section_alternating}

In this section we consider a refinement of Question \ref{first_q}.

\begin{question}\label{question_fixed_G}
Let $G$ be any finite group. Is it possible to find a subset $T$ of $S$ and a positive integer $n$ such that $\langle S_n(T) \rangle=G$  for every $n\geq n_0?$
\end{question}

We give a negative answer to the previous question for some families of groups. 

We need some preliminary results. Firstly notice that, if  $\langle S_n(T)\rangle$ is isomorphic to a fixed group $G$ for every $n\geq n_0,$ the cardinality of $S_n(T)$  must be bounded above by the order of $G.$ 

The set of pattern avoiding permutations with limited growth rate are well studied. See e.g. \cite{AlbertGeometricGC,Albert2006PermutationCO,HombergerOnTE, Huczynska2006GridCA, Kaiser2003OnGR,Pantone2020GrowthRO,VatterSmallPC,VatterPERMUTATIONCO}, to cite only a few.

Now we recall some results about classes of permutations with limited growth rate. First of all we introduce some definitions, following \cite{HombergerOnTE}.

An \textit{interval} in a permutation is a sequence of contiguous entries whose values form
an interval of natural numbers. A \textit{monotone interval} is an interval in which the entries are increasing or decreasing. Given a permutation $\sigma$ of length $m$ and nonempty permutations $\alpha_1,\ldots,\alpha_m,$ the \textit{inflation} of $\sigma$ by $\alpha_1,\ldots,\alpha_m$ is the permutation $\pi=\sigma[\alpha_1,\ldots,\alpha_m]$ obtained by replacing the $i$-th entry of $\sigma$ by an interval that is order isomorphic to $\alpha_i$, while maintaining the relative
order of the intervals themselves. For example,
$$3142[1, 321, 1, 12] = 6\; 321\; 7\; 45.$$
Notice that we allow  inflations by the empty permutation.

A \textit{peg permutation} $\tilde \rho$ is a permutation $\rho$ where some elements are labeled by the symbol $+$ and some other elements are labeled by the symbol $-.$
For example $\tilde \rho =3^{+}1^{-}24^{-}$ is a peg permutation whose underlying (non-pegged) permutation is $3124.$

The \textit{grid class} of the peg permutation $\tilde \rho$, where $\rho=\rho_1\ldots\rho_n,$ denoted by $\Grid(\tilde \rho)$, is the set of all permutations which may
be obtained inflating $\rho$ by monotone intervals of type determined by the signs of $\tilde \rho:$ $\rho_i$ may be
inflated by an increasing (resp., decreasing) possibly empty interval if $ \rho_i$ is labeled with a $+$ (resp., $-$) while it
may only be inflated by a single entry (or the empty permutation) if $ \rho_i$ is not labeled. 

For example the permutation $45621387$ is an element of the grid class $\Grid(3^+1^-24^-).$

Given a set $\tilde P$ of peg permutations, we denote the union of their corresponding grid classes by $$\Grid(\tilde P)=\cup_{\tilde \rho \in P}\Grid(\tilde \rho).$$

We will need the following result (\cite[Thm. 1.3]{HombergerOnTE}, which in turn is a consequence of some results in \cite{AlbertGeometricGC} and in \cite{Huczynska2006GridCA}).

\begin{thm}
For a permutation class $\mathcal{C}$ the following are equivalent:
\begin{enumerate}
    \item The cardinality $|\mathcal{C}_n|$ is a polynomial in $n,$ for sufficiently large $n.$
    \item $|\mathcal{C}_n|<F_n$ for some $n,$ where $F_n$ is the $n$-th Fibonacci numbers.
    \item $\mathcal{C}=\Grid(\tilde P),$ where $\tilde P$ is a finite set of peg permutations.
\end{enumerate}
\end{thm}

\begin{example}

Consider the class $\mathcal{C}=S(123, 231, 312).$ As we will see in the proof of Theorem \ref{threeofthree}, 
$$S(123, 231, 312)=\{k\quad k-1\;\ldots\;2\;1\;n\quad n-1\;\ldots\;k+1,\quad 1\leq k\leq n\}.$$

%The cardinality of $\mathcal{C}_n$ is clearly equal to $n,$ for every $n.$
The class $\mathcal{C}$ can be written as $\mathcal{C}=\Grid(\tilde P),$ where $\tilde P$ contains only the peg permutation $1^-2^-.$ 
\end{example}

An immediate consequence of the previous Theorem is the following Proposition that describes the structure of permutation classes with growth rate bounded above by a constant.

\begin{prop}\label{bounded_classes_structure}
For a permutation class $\mathcal{C}$ the following are equivalent.
\begin{enumerate}
    \item For every $n\geq 1,$ $|\mathcal{C}_n|<K$ where $K$ is a constant independent of $n.$
    \item $\mathcal{C}=\Grid(\tilde P),$ where $\tilde P$ is a finite set of peg permutations each of which has at most one labeled element.
     \item For sufficiently large $n,$ $|\mathcal{C}_n|=D,$ where $D$ is a constant independent of $n.$ 
    \end{enumerate}
\end{prop}
\proof
\begin{itemize}
    \item[] $[\textit 1.\implies \textit 2.]$  
    If for every $n\geq 1,$ $|\mathcal{C}_n|<K,$ where $K$ is a constant independent of $n,$ the growth rate of $|\mathcal{C}_n|$ is polynomial (the constant polynomial $K$). Hence, by the previous Theorem, $\mathcal{C}=\Grid(\tilde P)$ where $\tilde P$ is a finite set of peg permutations. Suppose that in $\tilde P$ there is a peg permutation $\tilde \rho$  with more than one element labeled. Let $a$ and $b$ be two labeled elements in $\tilde \rho.$ We can inflate $a$ and $b$ by increasing or decreasing sequences of arbitrary length. Hence $\Grid(\tilde P)\cap S_n$ would have at least $n+1$ elements of the form $t_1\,t_2\,\ldots\, t_k\,s_1\,s_2\,\ldots \,s_{n-k},$ with $0\leq k\leq n,$ where $t_1\,t_2\,\ldots\, t_k$ is the monotone sequence obtained by inflating $a,$ $s_1\,s_2\,\ldots \,s_{n-k}$ is the monotone sequence obtained by inflating $b,$ and all the other elements of $\rho$ have been inflated by the empty interval. This is a contradiction, because the cardinality of $\Grid(\tilde P)\cap S_n$ is bounded above by the constant $K.$
    
    \item[] $[\textit 2.\implies \textit 3.]$

    If $\mathcal{C}=\Grid(\tilde P)$ where $\tilde P$ is a finite set of peg permutations each of which has at most one labeled element, then $\mathcal{C}_n=\Grid(\tilde P)\cap S_n$ has a constant number of elements, for $n$ sufficiently large. In fact, let $\tilde \rho=a_1\ldots a_r b^+ c_1\ldots c_s $ an element of $\tilde P.$ If $n\geq r+s,$ such an element gives rise to exactly $2^{r+s}$ in $\mathcal C_n$ depending on how many of the $a_i'$s or $c_j'$s are inflated by the empty permutation.  The same holds when the element $b$ has label $-.$

    \item[] $[\textit 3.\implies \textit 1.]$
     Trivial.
\end{itemize}

\endproof

\begin{example}

Consider the class $\mathcal{C}=S(132, 312, 321, 2314).$ 

It is easily seen that 
$$S_n(132, 312, 321, 2314)=\{2\;3\;\ldots\;n\;1,\quad 2\;1\;3\;\ldots\;n,\quad \id_n\}$$ for every $n\geq 3.$

Hence $|\mathcal{C}_n|=3$ for $n\geq 3.$ The class $\mathcal{C}$ can be written as $\mathcal{C}=\Grid(\tilde P)$ where $\tilde P=\{213^+,\,2^+1\}.$

Notice also that the group generated by $S_n( 132, 312, 321, 2314)$ is $S_n$ by Lemma \ref{gen_sets}.

\end{example}

\begin{thm}\label{fixed_G_structure}
Let $T\subseteq S$ and let $G$ be a finite group. Suppose that there exists $n_0$ in $\mathbb N$ such that $\langle S_n(T)\rangle =G$ for every $n\geq n_0.$

Then $S(T)=\Grid(\tilde P),$ where every peg permutation in $\tilde P$ is either unlabeled, or of the form 
$$\tilde \rho=a_1\ldots a_r \; (r+1)^+\;b_1\dots b_s,$$
or 
$$\tilde \rho=b_1\dots b_s \; (r+1)^-\;a_1\ldots a_r,$$
where, for every $i$ and $j,$ $a_i<r+1$ and $b_j>r+1.$

In particular, for $n$ sufficiently large, every $\pi\in S_n(T)$ is either of the form 
$$\pi=123[\tau, \id,\sigma]$$
or of the form
$$\pi=321[\sigma,\psi,\tau]$$
where $\tau \leq \,a_1\ldots a_r,$ $\sigma \leq\,b_1\ldots b_s$ and $|\tau|+|\id|+|\sigma|=n.$
% $$\pi=a_1\ldots a_r \; r+1\;r+2\;\ldots \; n-s\;b_1\dots b_s$$ or of the form
% $$\pi=b_1\ldots b_s \;n-s \;n-s-1\;\ldots \; r+1\;a_1\dots a_r$$
% where $a_1\,\ldots \, a_r$ is a permutation of the symbols $\{1,\,\ldots\,, r\}$ and $b_1\,\ldots \, b_s$ is a permutation of the symbols $\{n-s+1,\,\ldots\, , n\},$ with  $0\leq r<n-s+1\leq n+1.$

\end{thm}

\proof
Since $\langle S_n(T)\rangle =G$ for every $n\geq n_0,$ the cardinality of $S_n(T)$ is bounded above by a constant $|G|=K$ independent of $n.$
By the previous corollary, each element of the class $ S(T)$ can be obtained inflating a peg permutation with at most one labeled element.

Let $\tilde \rho$ be a peg permutation with exactly one labeled element.  
\begin{itemize}
    \item Suppose $$\tilde \rho=a_1\ldots a_r \; c^+\;b_1\dots b_s,$$ where $c$ is the only labeled element, and suppose that there exists $1\leq i\leq r$ with $a_i>c.$ Then, inflating such $a_i$ by the permutation $1,$ every other $a_j$ and $b_j$ by the empty permutation and $c$ by the increasing permutation ($\id$), we obtain the element $n\;1\;2\;\ldots\;n-1\in S_n(T)$ for every $n.$ But this is a cycle of order $n.$ Hence $\langle S_n(T)\rangle$  contains at least $n$ elements  and its cardinality is not constant.

    Similarly, if there exists $1\leq i\leq s$ with $b_i<c,$ the set $S_n(T)$ contains the cycle $(1,\;2,\;3,\;\ldots\; ,n)$ of order $n.$ 

    \item Suppose otherwise that $$\tilde \rho=b_1\dots b_s \; c^-\;a_1\ldots a_r,$$ where $c$ is the only labeled element, and suppose that there exists $1\leq i\leq s$ with $b_i<c.$ Then, inflating such $b_i$ by the empty permutation or by the permutation $1,$ every other $b_j$ and $a_j$ by the empty permutation and $c$ by a decreasing permutation ($\psi$), we obtain that the elements $\psi_n$ and $\alpha=1\;n\quad n-1\;\ldots\;2\in S_n(T)$ for every $n.$ But $\alpha  \psi_n=(1,\;2,\;\ldots \; ,n),$ a cycle of order $n,$ and hence $\langle S_n(T)\rangle$ contains at least $n$ elements  and its cardinality would not be constant. 

    Similarly, if there exists $1\leq i\leq r$ with $a_i>c,$ the set $S_n(T)$ contains the permutations $\psi_n$ and $\beta=n-1\quad n-2\;\ldots\; 1\;n$ whose composition $\psi_n \beta=(1,\;2,\;\ldots\;,n)$ is a cycle of order $n.$ 
    
\end{itemize}
This concludes the proof of the first assertion. 

Now we consider the second assertion. If $n$ is greater than the maximal length of every unlabeled permutation in $\tilde P,$ every element in $S_n(T)$ can be obtained by inflating a peg permutation of one of the two forms above. 
Consider $$\tilde \rho=a_1\ldots a_r \; (r+1)^+\;b_1\dots b_s.$$
Let $\tau$  be a pattern in $a_1\ldots a_r$ and $\sigma$ be a pattern in $b_1\ldots b_s.$
Now inflate every element of  $\tau$ and   $\sigma$ by the permutation $1,$  the element $r+1$ by an increasing permutation and every other element by the empty permutation. In this way, we get the permutation $\pi=123[\tau,\id,\sigma].$

The permutation $\pi=321[\sigma,\id,\tau]$ can be obtained similarly from a peg permutation of the form 
$$\tilde \rho=b_1\dots b_s \; (r+1)^-\;a_1\ldots a_r.$$
\endproof

\begin{example}
    Consider the class $S(123,132,231,3214).$ As we will see in Theorem \ref{threethreeonefour} , 
     $\langle S_n(123,132,231,3214)\rangle \cong  (S_3\times S_3) \rtimes \mathbb Z_2,$ for every $n\geq 6,$  a group of order $72.$

    The permutations in $S_n(123,132,231,3214),$ with $n\geq 3$ are
    $\psi,$ $$\alpha=n\quad n-1\quad n-2\ldots \;4\;2\;1\;3$$ and $$ \beta=n\quad n-1\quad n-2\ldots \;4\;3\;1\;2.$$

    Hence $S(T)=\Grid(\tilde P),$ where $\tilde P$ consists only of the peg permutation $4^-213.$
\end{example}

Now we are in position to answer Question \ref{question_fixed_G} for every abelian $G.$

\begin{thm}
Let $T\subseteq S$ and let $G$ be a finite, nontrivial abelian group. If there exists $n_0$ in $\mathbb N$ such that $\langle S_n(T)\rangle \cong G$ for every $n\geq n_0,$ then $G\cong \mathbb Z_2$ or $G\cong \mathbb Z_2 \times \mathbb Z_2.$
\end{thm}
\proof
Suppose that $S_n(T)$ generates a fixed abelian group $G$ for every $n\geq n_0.$

By Proposition \ref{bounded_classes_structure} and by Theorem \ref{fixed_G_structure}, $\cup_{n\geq 0} S_n(T)=\Grid(\tilde P),$ where $\tilde P$ is a finite set of peg permutations each of which is either unlabeled, or of the form
$$\tilde \rho=a_1\ldots a_r \; (r+1)^+\;b_1\dots b_s,$$
or 
$$\tilde \rho=b_1\dots b_s \; (r+1)^-\;a_1\ldots a_r,$$
where, for every $i$ and $j,$ $a_i<r+1$ and $b_j>r+1.$

\begin{itemize}
    \item[\textbf{Case 1}] Suppose that in $\Grid(\tilde P)$ there exists a peg permutation $$\tilde \rho=a_1\ldots a_r \; (r+1)^+\;b_1\dots b_s.$$

\begin{itemize}
\item If $a_1\,\ldots \, a_r$ contains a subsequence $a_x,a_y,a_z$ order isomorphic to $132,$ then, inflating $a_x,$ $a_y$ and $a_z$ by the permutation $1,$ every other $a_i$ and $b_j$ with the empty permutation and $c$ with an increasing permutation, we obtain in $S_n(T)$ the permutation $1\,3\,2\,4\,5\,\ldots\,n=(2,\;3).$

Inflating $a_y$ and $a_z$ by the permutation $1,$ every other $a_i$ and $b_j$ with the empty permutation and $c$ with an increasing permutation we obtain in $S_n(T)$ the permutation $2\,1\,3\,4\,\ldots\,n=(1,\;2).$ 
But this is impossible since $(1,\;2)$ and $(2,\;3)$ do not commute. 
\item Reasoning as above, if $a_1\,\ldots \, a_r$ contains the pattern $321$ we get the cycles $(1,\;2)$ and $(1,\;3)$ that do not commute.
\item  If $a_1\,\ldots \, a_r$ contains the pattern $231$ we get the cycles $(1,\;2)$ and $(1,\;2,\;3)$ that do not commute.
\item  If $a_1\,\ldots \, a_r$ contains the pattern $312$ we get the cycles $(1,\;2)$ and $(1,\;3,\;2)$ that do not commute.
\end{itemize}

As a consequence, $a_1\,\ldots \, a_r \in S_r(132,231,312,321)$ and hence either $$a_1\,\ldots \, a_r=\id_r,$$ or $$a_1\,\ldots \, a_r=2\;1\;3\;4\;\ldots\;r.$$

Now we turn our attention to the suffix $b_1\,\ldots\,b_s.$

\begin{itemize}
    \item If $b_1\,\ldots \, b_s$ contains the  subsequence $b_x,b_y,b_z$ order isomorphic to  $321,$ then, inflating $b_x,$ $b_y$ and $b_z$ by the permutation $1,$ every other $a_i$ and $b_j$ with the empty permutation and $c$ with an increasing permutation, we obtain in $S_n(T)$ the transposition $(n-2,\;n).$

    Inflating $b_y$ and $b_z$ by the permutation $1,$ every other $a_i$ and $b_j$ by the empty permutation and $c$ by an increasing permutation we obtain in $S_n(T)$ the transposition $(n-1,\;n).$
    But this is impossible since $(n-2,\;n)$ and $(n-1,\;n)$ do not commute. 

    \item Similarly, if $b_1\,\ldots \, b_s$ contains the pattern $231$ we get the cycles $(n-1,\;n)$ and $(n-2,\;n-1,\;n)$ that do not commute.

    \item If $b_1\,\ldots \, b_s$ contains the pattern $213$ we get  the cycles $(n-1,\;n)$ and $(n-2,\;n-1)$ that do not commute.

    \item If $b_1\,\ldots \, b_s$ contains the pattern $312$ we get  the cycles $(n-1,\;n)$ and $(n-2,\;n,\;n-1)$ that do not commute.
    
\end{itemize}

As a consequence the suffix $b_1\,\ldots \, b_s$ of $\tilde\rho$ is a permutation of the symbols $\{r+2,\ldots ,r+s+1\}$ that avoids $213,231,312,$ and $321$ and hence it is easily seen that $b_1\,\ldots \, b_s$ is either the identity or the transposition that exchanges the two last symbols. 

We deduce that the peg permutation $\tilde \rho$ is one of the following $$1^+,\; 213^+,\;1^+32,\;213^+54.$$

\item[\textbf{Case 2}] Suppose that in $\Grid(\tilde P)$ there exists a peg permutation $$\tilde \rho=b_1\dots b_s \; (r+1)^-\;a_1\ldots a_r.$$

Inflating every $a_i$ and every $b_j$ by the empty permutation and $r+1$ by a decreasing permutation, we get the permutation $\psi_n\in S_n(T).$  Notice that every other permutation $\alpha$ in $S_n(T)$ must commute with $\psi_n$ because we are supposing that $S_n(T)$ generates a commutative group. But $\alpha \psi_n=\psi_n \alpha$ is equivalent to $\alpha=\alpha^{rc}.$ Hence every element of $S_n(T)$ must be equal to its reverse-complement. 

Now consider once more the peg permutation $$\tilde \rho=b_1\dots b_s \; (r+1)^-\;a_1\ldots a_r.$$  If $b_1\dots b_s$ contains the pattern $12$ realized by the elements $b_x$ and $b_y$ (with $x<y$), inflating $b_x$ and $b_y$ by the permutation $1,$ every other $b_i$ and $a_j$ by the empty permutation and $r+1$ by a decreasing permutation, we get the permutation $\alpha=n-1\quad n\quad n-2\quad n-3\;\ldots 1\in S_n(T).$ This is impossible since $\alpha\neq \alpha^{rc}.$ As a consequence $b_1\dots b_s$ and, similarly, $a_1\dots a_r$ must be the decreasing permutations and hence the peg permutation $\tilde \rho$ is equivalent to the peg permutation $1^-.$
\end{itemize}

\textbf{Case 1} and \textbf{Case 2} prove that $S(T)=\Grid(\tilde P),$ where $\tilde P$ contains either unlabeled permutations or elements from the set of peg permutation $$\{1^-,\;1^+,\; 213^+,\;1^+32,\;213^+54\}.$$

It is easily checked that if $\tilde P$ contains $1^-$ then $\tilde P$ does not contain $213^+,$ $1^+32,$ $213^+54$, since $G$ is commutative. 
For every possible choice of $\tilde P$ we can only get either $G\cong \mathbb Z_2$ or $G \cong \mathbb Z_2\times \mathbb Z_2$ for $n$ large.

This concludes the proof. 
\endproof

Now we consider Question \ref{question_fixed_G} when $G$ is an alternating group.

\begin{thm}
Let $T\subseteq S$ and let $G=A_m,$ $m\geq 1,$ be an alternating group. If there exists $n_0$ in $\mathbb N$ such that $\langle S_n(T)\rangle \cong G$ for every $n\geq n_0,$ then $m\leq 2.$
\end{thm}

\proof
Suppose that $S_n(T)$ generates a fixed alternating group $G=A_m$ for every $n\geq n_0.$

Proposition \ref{bounded_classes_structure} and  Theorem \ref{fixed_G_structure} imply that $\cup_{n\geq 0} S_n(T)=\Grid(\tilde P)$ where $\tilde P$ is a finite set of peg permutations each of which is either unlabeled, or of the form
$$\tilde \rho=a_1\ldots a_r \; (r+1)^+\;b_1\dots b_s,$$
or 
$$\tilde \rho=b_1\dots b_s \; (r+1)^-\;a_1\ldots a_r,$$
where, for every $i$ and $j,$ $a_i<r+1$ and $b_j>r+1.$

If one peg permutation of the second kind appears in $\tilde P,$ inflating each $a_i$ and $b_j$ by the empty permutation and $r+1$ by a decreasing permutation we get $\psi_n\in S_n(T)$ for every $n.$ But in this case it is impossible that $S_n(T)$ generates a fixed alternating group $A_m$ for sufficiently large $n,$ because $\psi_n$ is odd if $n$ is odd. Hence every labeled peg permutation in $\tilde P$ is of the first kind. 

Let $$\tilde \rho=a_1\ldots a_r \; (r+1)^+\;b_1\dots b_s$$ be one of the labeled peg permutation in $\tilde P.$

Suppose that $a_1\ldots a_r$ contains the pattern $21,$ realized by the elements $a_x$ and $a_y$ (with $x<y$). Then, inflating $a_x$ and $a_y$ by the permutation $1,$ every other $a_i$ and $b_j$ by the empty permutation and $r+1$ by the identity, we get the permutation $2\,1\,3\,\ldots\,n=(1,\;2)\in S_n(T).$ This is impossible, since this permutation is odd. 
Hence $a_1\ldots a_r=\id_r.$

We can show similarly that $b_1\ldots b_s$ must be the increasing permutation $\id_s.$

As a consequence $\tilde \rho$ is equivalent to $1^+,$  the set $S_n\cap\Grid(\tilde P)$ contains only $\id_n$ for sufficiently large $n$ and $\langle S_n(T)\rangle= \{\id_n\}=A_1=A_2.$
\endproof

\section{Sets $T$ such that $\langle S_n(T) \rangle=S_n.$}\label{section_gen_Sn}

From now on we determine the groups $\langle S_n(T) \rangle$ for particular sets of patterns $T.$

First of all we state a rather general result that allows us to determine sets $T\subseteq S$ such that $\langle S_n(T) \rangle=S_n.$

\begin{thm}\label{general_generatorsSn}
% Let $T\subseteq S_k$ where $k$ is a fixed positive integer. If 
% $$T\cap \{k\;2\;3\ldots k-1\;1,\quad k\;1\;2\ldots k-1,\quad 2\;3\ldots k\;1,\quad \id_k \}=\emptyset$$
% or
% $$T\cap \{\id_k,\quad(r-1\quad r)\mbox{ for }2\leq r\leq k\}=\emptyset$$
% or
% $$T\cap \{2\;3\ldots k\;1,\quad \id_k,\quad 2\;1\;3\ldots k \}=\emptyset$$
% or
% $$T\cap \{1\;\ldots k-2\;k\;k-1,\quad \id_k,\quad 2\;3\;\ldots k \;1\}=\emptyset$$
% than $\langle S_n(T)\rangle=S_n$ for every $n.$
Let $T\subseteq S$. If either
$$T\cap \{k\;2\;3\ldots k-1\quad 1,\quad k\;1\;2\ldots k-1,\quad 2\;3\ldots k\;1,\quad \id_k \;|\; k\geq 1\}=\emptyset,$$
or
$$T\cap \{\id_k,\quad(r-1,\quad r)\;|\; 2\leq r\leq k \mbox{ and }k\geq 1 \}=\emptyset,$$
or
$$T\cap \{2\;3\ldots k\;1,\quad \id_k,\quad 2\;1\;3\ldots k \;|\; k\geq 1\}=\emptyset,$$
or
$$T\cap \{1\;\ldots k-2\quad k\quad k-1,\quad \id_k,\quad 2\;3\;\ldots k \;1 \;|\; k\geq 1\}=\emptyset,$$
then $\langle S_n(T)\rangle=S_n$ for every $n.$

In particular,
\begin{itemize}
\item[i.]
if $T\subseteq S_3$ and $\langle S_n(T)\rangle\neq S_n,$ then $T$ shares at least an element with each one of the following sets $$\{123, 231, 312, 321\},\; \{123,132,213\},\;\{123,213,231\},\mbox{  and  }$$ $$\{123,132,231\}.$$ 
\item[ii.] if $T\subseteq S_4$ and $\langle S_n(T)\rangle \neq S_n,$ then $T$ shares at least an element with each one of the following sets $$\{1234, 2341, 4123, 4231\},\;\{1234,1243,1324,2134\},$$ $$\{1234,2134, 2341\}\mbox{  and  }\{1234,1243,2341\}.$$
\end{itemize}
\end{thm}

\proof
If $n<k$ the theorem is trivially true. 
Assume that $n>k.$
Note that
\begin{itemize}
\item the patterns of length $k$ contained in at least one elementary transposition in $S_n$ are  $\id_k$ and $\quad(r-1,\, r)\mbox{ for }2\leq r\leq k,$
\item the patterns of length $k$ contained either in the cycle $(1,\;n)\in S_n$ or in $(1,\;2,\;\ldots \;,n)\in S_n$ are 
$k\;2\;3\ldots k-1\quad 1,\quad k\;1\;2\ldots k-1,\quad 2\;3\ldots k\;1,\quad \id_k,$
\item the patterns of length $k$ contained either in the cycle  $(1,\;2)\in S_n$ or in $(1,\;2,\;\ldots \;,n)\in S_n$ are $2\;3\ldots k\;1,\quad \id_k,\quad 2\;1\;3\ldots k,$ and 
\item the patterns of length $k$ contained either in the cycle   $(n-1,\;n)\in S_n$ or in $(1,\;2,\;\ldots \;,n)\in S_n$ are $2\;3\ldots k\;1,\quad \id_k,\quad 1\;2\ldots k-2\quad k\quad k-1.$
\end{itemize}

The assertion follows by Lemma \ref{gen_sets}.

If $k=n$ the proof is similar. 
\endproof

\section{Groups generated by $S_n(T)$ with $T\subseteq S$ and $|T|\leq 3.$}\label{section_T_three}

In this section we consider the group generated by the set $S_n(T)$, where $T\subseteq S$ and $|T|\leq 3.$ 

We begin with the analysis of $\langle S_n(T)\rangle$ when $T$ consists of  three patterns of length 3.

\begin{thm}
\label{threeofthree}
\begin{enumerate}
    \item  $\langle S_n(123,321,\tau) \rangle = \{\id_n\},$ for every $n>4$ and for every $\tau\in S_3.$
  %  \item $\langle S_n(123,132,213) \rangle = S_n,$
    \item $\langle S_n(132,213,321) \rangle \cong \mathbb{Z}_n,$ 
    % \item $\langle S_n(132,312,321) \rangle = \langle S_n(132,231, 321) \rangle = \langle S_n(213,231, 321) \rangle =$ 
   % $\langle S_n(213, 312, 321) \rangle = S_n,$
    % \item $\langle S_n(123, 213, 231)\rangle =\langle S_n(123, 213, 312) \rangle  =  \langle S_n(123, 132, 312) \rangle =$  
    
    % $\langle S_n(123, 132, 231) \rangle = S_n,$
    % \item $\langle S_n(213,231,312) \rangle = \langle S_n(132,231,312)  \rangle = S_n,$
    % \item $\langle S_n(132,213,231) \rangle = \langle S_n(132,213,312)  \rangle = S_n,$
    %     \item $\langle S_n(231,312,321) \rangle = S_n,$
   \item $\langle S_n(123,231,312) \rangle \cong D_n,$ for every $n\geq 3$.
    \item $\langle S_n(T)\rangle =S_n$ for every other set $T$ of three patterns of length three.
\end{enumerate}
\end{thm}
\begin{proof}
    
\begin{enumerate}
    \item  The set of generators is trivially empty for $n>4.$
    % \item The set of generators contains $\psi$ and $(i\ i+1) \psi$ for every $i$. Then, the generated group contains all elementary transpositions.
    \item The set of generators (which turns out to be itself a group by Lemma \ref{group_patterns}) consists of the cycle $(1,\,2,\,\ldots\,,n)$ and its powers.
%     \item The set of generators contains the cycle $(1\,2\,\ldots\,n)$ and the transposition $(1\, 2)$.
%     \item The set $S_n(123, 213, 231)$ contains $\psi$ and the permutations:
%     $$\alpha=1\,\,n\,\,n-1\,\,\ldots\,\,3\,\,2,$$
%     $$\beta=n\,\,1\,\,n-1\,\,n-2\,\,\ldots\,\,3\,\,2.$$ Now, straightforward computations show that $\alpha\beta=(1\, 2)$ and 
%  $\alpha\psi=(1\,2\,\ldots\,n)$. 
%  \item The set $S_n(213,231,312)$ contains only $n$ elements whose on line notation is $$\sigma \,=\, 1\;2\; \ldots\, k\; n \; n-1\;\ldots\, k+2\; k+1,\quad 1\leq k\leq n.$$
%  Hence,  Then, $S_n(213,231,312)$ contains the transposition $(n-1\;n)$, $\psi$ and $$\alpha\,=\,1\,\,n\,\,n-1\,\,\ldots\,\,3\,\,2.$$
% We deduce that $\langle S_n(213,231,312)\rangle = S_n.$ 
% \item The set $S_n(132,213,231)$ contains only $n$ elements whose on line notation is $$n \; n-1\; k+2\; k+1\; 1\; 2 \; \ldots \; k,\quad 1\leq k\leq n.$$
%  Hence,  Then, $S_n(132,213,231)$ contains $\psi$ and $$\gamma\,=\,n\,\,n-1\,\ldots\,3\, 1\,2.$$
%  Now, straightforward computations show that $\psi\gamma=(n-1\; n)$.
% Since $S_n(132,213,231)$ contains $(n-1\; n)$ and $\alpha$ defined above, we deduce that $\langle S_n(132,213,231)\rangle = S_n.$
% \item The set $S_n(231,312,321)$ contains all the elementary transpositions. 
\item The set $S_n(123,231,312)$ contains only $n$ elements whose one-line notation is:
$$k\quad k-1\;\ldots\;2\;1\;n\quad n-1\;\ldots\;k+1,\quad 1\leq k\leq n.$$
This means that $\psi$ and 
$$\alpha\,=\,1\,\,n\quad n-1\,\,\ldots\,\,3\,\,2$$
belong to this set. Observe that 
$$\alpha^2=\psi^2=(\alpha\psi)^n=\id,$$
hence $$\langle \alpha, \psi\rangle \cong D_n.$$
It is easily seen that the other elements in $S_n(123,231,312)$ are of the form $(\alpha \psi)^j \alpha$ for some $j$. 

\item By Lemma \ref{rcifnotpsi} we have 
$$\langle S_n(123,132,213)\rangle=\langle S_n(123,132,213) \cup S_n(231,312,321)\rangle,$$
and $\langle S_n(231,312,321)\rangle=S_n$ by Theorem \ref{general_generatorsSn}[\textit{i.}].

By Lemma \ref{rc_inv} we have 
$$\langle S_n(123, 213, 231)\rangle \cong\langle S_n(123, 213, 312) \rangle  \cong  \langle S_n(123, 132, 312) \rangle \cong$$
$$\langle S_n(123, 132, 231) \rangle,$$ and 
$$\langle S_n(132,312,321) \rangle \cong \langle S_n(132,231, 321) \rangle \cong \langle S_n(213,231, 321) \rangle \cong$$    
    $$\langle S_n(213, 312, 321) \rangle,$$ 
 by Lemma \ref{rcifnotpsi} we have
$$\langle S_n(123, 213, 231)\rangle=$$ $$\langle S_n(123, 213, 231)\cup S_n(132, 312,321)\cup  S_n(213, 231, 321)\cup S_n(123,132,312)\rangle,$$
and $\langle  S_n(132, 312, 321)\rangle=S_n$ by Theorem \ref{general_generatorsSn}[\textit{i.}].

By Lemma \ref{rc_inv} we have  
 $$\langle S_n(213,231,312) \rangle \cong \langle S_n(132,231,312)  \rangle,$$
and
$$\langle S_n(132,213,231) \rangle \cong \langle S_n(132,213,312)  \rangle,$$
by Lemma \ref{rcifnotpsi} we have 
$$\langle S_n(213,231,312)\rangle =\langle S_n(132,213,231) \rangle= $$
$$\langle S_n(213,231,312) \cup S_n(312,132,213)\cup$$ $$ S_n(132,213,231)\cup S_n(132,231,312)\rangle.$$
This last generating set contains the transposition $(1,\;2)$ (an element of $S_n(132,231,312)$) and the cycle $(1,\;2,\;\ldots \;,n)$ (an element of  $S_n(132,213,312)$). Hence the generated group is $S_n$ by Lemma \ref{gen_sets}.
 \end{enumerate}
\end{proof}

As a consequence we can characterize all the groups $\langle S_n(T) \rangle$ where $T$ contains no more than three arbitrary patterns. 

\begin{cor}\label{lessthanthreepatterns}
Let $T\subseteq S$ be a set of patterns, with $|T|\leq 3.$
\begin{enumerate}
    \item If $ \{ \id_r,\psi_s\}\subseteq T,$ for some $r,s \geq 1,$ then $\langle S_n(T)\rangle = \{ \id_n\}$ for $n>(r-1)(s-1).$ 
    \item If for every $r,s\geq 1$ $\{ \id_r,\psi_s\}\nsubseteq T,$ $T\neq \{132,213,321\}$ and $T\neq \{123,231,312 \},$ then $\langle S_n(T)\rangle = S_n.$ 

\end{enumerate}
\end{cor}

\proof
The  first assertion is trivial since the set $S_n(T)$ is empty for $n$ sufficiently large  when $\{ \id_r,\psi_s\}\subseteq T.$
More precisely, every permutation of length $n>(r-1)(s-1)$ must contain a monotonically increasing subsequence of length $r$ or a monotonically decreasing subsequence of length $s,$ by the Erd\H os-Szekeres Theorem \cite{Erd_Sz}.

To prove the second assertion, recall that the only sets $T$ of three patterns of length three  such that $\langle S_n(T)\rangle \neq S_n$ are $\{132,213,321\}$ and $\{123,231,312\},$  by Theorem \ref{threeofthree}.  Hence, the second assertion follows from Theorem \ref{threeofthree} and from Lemmas \ref{inclusion} and \ref{sottopatterns} in all cases, except for the following ones.
\begin{itemize}
    \item \textbf{Case 1} $T$ consists of three (distinct) patterns $\tau_1,\tau_2$ and $\tau_3$ with $\tau_1\geq 132,$ $\tau_2\geq 213$ and $\tau_3\geq 321,$ with at least one of these patterns of length strictly greater than 3.
    \item \textbf{Case 2} $T$ consists of three (distinct) patterns $\tau_1,\tau_2$ and $\tau_3$ with $\tau_1\geq 123,$ $\tau_2\geq 231$ and $\tau_3\geq 312,$ with at least one of these patterns of length strictly greater than 3.
\end{itemize}

If Case 1 occurs, we can assume that $\tau_3\neq 321,$ otherwise Lemma \ref{sottopatterns} and Theorem \ref{threeofthree} would imply that $\langle S_n(T)\rangle =S_n.$

Notice that $\langle S_n(132,213,4321)\rangle =S_n$ because the set of generators contains the cycles $(1,\; 2,\, \ldots\,,n )$ and $(1,\, n),$ which generate $S_n$ by Lemma \ref{gen_sets}.

Now the fact that $S_n(132,213,\tau_3)$ with $\tau_3> 4321$  generates $S_n$ follows by Lemma \ref{sottopatterns}.

If Case 2 occurs, we can assume that $\tau_1\neq 123,$ otherwise Lemma \ref{sottopatterns} and Theorem \ref{threeofthree}  would imply that $\langle S_n(T)\rangle=S_n.$

Notice that $$\langle S_n(231,312,1234)\rangle=$$ 
$$\langle S_n(231,312,1234)\cup S_n(132,213,4321) \cup S_n(231,312,1234)\rangle $$ by Lemma \ref{rcifnotpsi}. This last group is $S_n$ because $S_n(132,213,4321)$ contains the permutations $(1,\;n)$ and $(1,\;2,\;\ldots \;,n),$ which generate $S_n$ by Lemma \ref{gen_sets}.

Now the fact that $S_n(\tau_1,231,312)$ with $\tau_1> 1234$ generates $S_n$ follows by Lemma \ref{sottopatterns}.

\endproof

\section{Groups generated by $S_n(T)$ with $T\subseteq S_3 \cup S_4$ and $|T|\geq 4.$}\label{section_T_four}

In this section we begin the study of the groups $\langle S_n(T)\rangle$ where $T$ has cardinality greater than three. We select some of the most meaningful cases. 

We consider the group generated by $S_n(T)$ with $T\subseteq S_3.$ The only cases not already covered by Theorem \ref{threeofthree} and  Corollary \ref{lessthanthreepatterns} are those with $|T|\geq 4.$
The following theorem deals with all these cases.

\begin{thm}
\begin{enumerate}
\label{fourofthree}
    \item $\langle S_n(\{123,321\}\cup T') \rangle = \{\id\},$ for every $n>4,$ where $T'$ is an arbitrary subset of $S.$
    \item $\langle S_n(213,231,312,321) \rangle \cong \langle S_n(132,231,312,321) \rangle \cong \langle S_n(132,213,231,312) \rangle\cong \mathbb{Z}_2,$
    \item $\langle S_n(123,132,213,231) \rangle \cong \langle S_n(123,132,213,312) \rangle \cong D_4,$
    \item $\langle S_n(132,213,312,321) \rangle \cong \langle S_n(132,213,231,321) \rangle \cong \mathbb{Z}_n,$
    \item $\langle S_n(123,213,231,312) \rangle \cong \langle S_n(123,132,231,312) \rangle \cong D_n,$
    \item $\langle S_n(123,132,213,231,312) \rangle \cong \mathbb Z_2,$
    \item $\langle S_n(132,213,231,312,321) \rangle \cong \mathbb \{\id_n\}.$
\end{enumerate}
\end{thm}
\begin{proof}
    
\begin{enumerate}
    \item  The set of generators is trivially empty for $n>4.$
    \item In each case, we have that the set of generators consists only of the identity and an involution. 
    \item The set $S_n(123,132,213,231)$ contains only $\psi$ and $(1,\; 2)\psi$ (whose period is $4$), hence $\langle S_n(123,132,213,231)\rangle = D_4$.
    \item The set $S_n(132,213,312,321)$ contains only the identity and the cycle $(1,\;2,\;3,\;\ldots\;,n).$ 
    \item The set $S_n(123,213,231,312)$ contains only $\psi$ and $\alpha=1\;n\quad n-1\;\ldots \;3\;2$. As seen in the proof of Theorem \ref{threeofthree}, these two permutations generate the dihedral $D_n$.
    \item The only element in the generating set is $\psi_n.$
    \item The only element in the generating set is $\id_n.$
\end{enumerate}
\end{proof}

 Now we consider the case of three patterns $\tau_1,\tau_2,\tau_3$ in $S_3$ and one pattern $\tau_4$ in $S_4.$ We suppose also that $\tau_4$ avoids each of $\tau_1,$ $\tau_2$ and $\tau_3,$ otherwise we fall back in the previous cases.
 % Notice that in the following Theorem we do not report all the possible cases. However the list is complete up to the application of the inverse or reverse-complement operators which, by Lemma \ref{rc_inv}, do not change the generated group up to isomorphism.

\begin{thm}
\label{threethreeonefour}
\begin{enumerate}
\item $\langle S_n(123,132,213,4231)\rangle \cong \langle S_n(132,213,231,4123) \rangle \cong \langle S_n(132,231,312,3214) \rangle\cong D_4$ for every $n\geq 4.$
\item $\langle S_n(123,132,231,3214)\rangle \cong  \langle S_n(132,213,231,1234)\rangle\cong  (S_3\times S_3) \rtimes \mathbb Z_2 $ for every $n\geq 6.$
\item $\langle S_n(123,231,312,1432) \rangle \cong \langle S_n(123,231,312,2143) \rangle \cong $

$\langle S_n(132,213,231,4312) \rangle \cong \langle S_n(132,231,312,2134) \rangle \cong D_n$ 

for every $n\geq 2.$

\item $\langle S_n(132,213,321,2341) \rangle \cong \langle S_n(132,213,321,3412) \rangle \cong \mathbb Z_n$

for every $n\geq 1.$

\item $\langle S_n(132,231,312,4321) \rangle \cong \langle S_n(132,231,321,4123) \rangle \cong  S_3$

for every $n\geq 1.$

\item $\langle S_n(231,312,321,1243) \rangle \cong \langle S_n(231,312,321,2134) \rangle \cong  S_4$

for every $n\geq 1.$

\item
$\langle S_n(231,312,321,1324)\rangle=\mathbb Z_2 \times \mathbb Z_2\cong D_2$ for every $n\geq 4.$

\item 
$\langle S_n(123,1 3 2, 213, 4 3 1 2)\rangle\cong (((((\mathbb Z_2\times \mathbb Z_2\times \mathbb Z_2\times \mathbb Z_2) \rtimes \mathbb Z_3) \rtimes \mathbb Z_2) \rtimes \mathbb Z_3 ) \rtimes \mathbb Z_2) \rtimes \mathbb Z_2$ for every $n\geq 8.$

\item
Let $\tau_1,\tau_2,\tau_3\in S_3$ and $\tau_4\in S_4,$ and suppose that the set $\{\tau_1,\tau_2,\tau_3,\tau_4\}$ is not obtained from one of the set in the preceding cases either by the application of the reverse-complement map, or of the inverse map, or their composition. Suppose also that $\tau_4$ avoids $\tau_1,$ $\tau_2$ and $\tau_3.$  Then, for sufficiently large $n,$
$\langle S_n(\tau_1,\tau_2,\tau_3,\tau_4)\rangle$ is either $S_n,$ or the trivial group $\{ \id_n\}$. 
\end{enumerate}
\end{thm}
\proof
\begin{enumerate}
\item We prove only that 
$\langle S_n(123,132,213,4231)\rangle \cong D_4,$ the other cases are analogous. Since a permutation in the generating set must avoid 123, 132, 213 and 4231, the only elements of this set are $\psi,$ $\alpha,$ $\beta$ and $\gamma,$ where $$\alpha=n-1\quad n\quad n-2\quad n-3\ldots \;4\;3\;1\;2\;=\;(1,\,n-1)(2,\,n )\prod_{j=3}^{\lfloor \frac{n}{2}\rfloor +1}(j,\,n-j+1),$$ $$\beta=n\quad n-1\quad n-2\quad n-3\ldots \;4\;3\;1\;2\;=\; (1,\,n,\,2,\,n-1) \prod_{j=3}^{\lfloor \frac{n}{2}\rfloor +1}(j,\,n-j+1),$$ and $$\gamma=n-1\quad n\quad n-2\quad n-3\ldots \;4\;3\;2\;1\;=\;(1,\, n-1,\,2,\, n)\prod_{j=3}^{\lfloor \frac{n}{2}\rfloor +1}(j,\,n-j+1).$$

Note that $\gamma=\beta^{-1},$  $\psi=\beta\alpha\beta,$ the order of $\beta$ is four,  the order of $\alpha$ and $\alpha\beta=(1,\;2)$ is two, and the group generated by $\psi,\alpha,\beta$ and $\gamma$ is $D_4.$

\item We prove only that $\langle S_n(123,132,231,3214)\rangle \cong  (S_3\times S_3) \rtimes \mathbb Z_2 .$ 

Since a permutation in the generating set must avoid 123, 132, 231 and 3214, the only elements of this set are $\psi,$ $\alpha$ and $\beta,$ where $$\alpha=n\quad n-1\quad n-2\ldots \;4\;2\;1\;3\;=\;(1,\,n,\,3,\,n-2,\,2,\,n-1)\prod_{j=4}^{\lfloor \frac{n}{2}\rfloor +1}(j,\,n-j+1),$$ and $$ \beta=n\quad n-1\quad n-2\ldots \;4\;3\;1\;2\;=\;(1,\,n,\,2,\,n-1)(3,\,n-2) \prod_{j=4}^{\lfloor \frac{n}{2}\rfloor +1}(j,\,n-j+1).$$

Notice that $\psi$ can be eliminated from the set of generators, since $$\psi=(\beta^{-1}\alpha^{-1})^3\beta^{-1}.$$
Now, the subgroup $\langle \alpha,\beta\rangle$ of $S_n$ is clearly isomorphic to the subgroup $G=\langle \widehat \alpha,\widehat \beta  \rangle$ of $S_6$ where $\widehat \alpha =(1,\,6,\,3,\,4,\,2,\,5)$ and $\widehat \beta=(1,\,6,\,2,\,5)(3,\,4).$

Consider now the subgroups $N$ and $H$ of $S_6$ defined by $$N=\langle (4,\,6,\,5),\, (2,\,3)(4,\,5),\, (1,\,3,\,2)(4,\,5,\,6),\, (1,\,6,\,3,\,4,\,2,\,5) \rangle$$ and $$H=\langle (5,,6) \rangle .$$ Long but trivial calculations show that $N$ and $H$ are subgroups of $G,$ with $N\unlhd G,$  every $g\in G$ can be written as $g=kh,$ where $k\in N$ and $h\in H,$ and  $N\cap H=\{ id\}.$ By Lemma 
\ref{sdp}, we have $G\cong N\rtimes H.$

Moreover, it is easy to show that $N\cong S_3\times S_3.$ In fact, if we identify $S_3\times S_3$ with the subgroup of $S_6$ generated by the cycles $(1,2,3),$ $(1,2),$ $(4,5,6)$ and $(4,5),$ the following map $f$ yields an isomorphism between $N$ and $S_3\times S_3.$  $$f((4,\,6,\,5))=(1,\,2,\,3)(4,\,6,\,5),\quad f((2,\,3)(4,\,5))=(1,\,2)(4,\,6),$$ $$f((1,\,3,\,2)(4,\,5,\,6))=(1,\,2,\,3), \mbox{ and }f((1,\,6,\,3,\,4,\,2,\,5))=(1,\,3,\,2)(4,\,5).$$

\item We prove only that $\langle S_n(123,231,312,1432) \rangle \cong D_n.$ 

The generating set is easily seen to contain only the permutations $\psi, $ $\alpha$ and $\beta,$ where $$\alpha= n-1\quad n-2\; \ldots \;2\; 1\;n\,=\,\prod _{j=1}^{\lfloor \frac{n}{2} \rfloor+1}(j,n-j) $$ and $$\beta= n-2\quad n-3\quad  \ldots \;2\; 1\;n\quad  n-1=\, (n-1,\,n)\prod _{j=1}^{\lceil  \frac{n}{2} \rceil}(j,\,n-j-1).$$

Notice that $\beta$ is a superfluous generator since $\beta=\alpha\psi\alpha,$ and that $\psi\alpha=(1,\,2,\,3,\,\ldots \, ,n)$ has order $n.$ The assertion follows. 

\item We consider only $\langle S_n(132,213,321,2341) \rangle.$ 

The only elements of the generating set different from the identity permutation are  $$\alpha=n\;1\;2\;\ldots \; n-1\,=\,(1,\,n,\,n-1,\,\ldots\,,2) $$ and $$\beta=n-1\quad n\,1\,\ldots \, n-2.$$ Notice that $\beta$ is composed by two cycles of length $\frac{n}{2}$ if $n$ is even and by one cycle of length $n$ is $n$ is odd. In both cases $\beta=\alpha^2,$ hence the  generated group is $\mathbb Z_n.$

\item This is a particular case of Theorem \ref{generatingSk}.
% The only elements in the generating set are the identity, the cycle $(1\,2)$ and the cycle $(1\,3)$ in the first case, and the identity, the cycle $(1\,2)$ and the cycle $(1\,3\,2)$ in the second case.

\item The only elements in the generating set are the identity,  $(1,\,2),$ $(2,\,3)$ and $(1,\,2)(3,\,4)$ in the first case, and the identity,  $(n-1,\,n),$ $(n-2,\,n-1)$ and $(n-3,\,n-2)(n-1,\,n)$ in the second case.

\item Given a permutation in the generating set, the symbols $n-1$ and $n$ must occupy the last and the second last position, otherwise the permutation contains 312 or 321. Similarly, the symbols 1 and 2 must occupy the first or the second position. The remaining symbols must be in increasing order since the permutation must avoid 1324.

Hence the set $S_n(231,321,312,1324)$ contains only the identity, the elementary transpositions $(1,\,2)$ and $(n-1,\,n)$, and the involution $(1,\,2)(n-1,\,n),$ which generate $\mathbb Z_2\times \mathbb Z_2.$ 

% \item Trivial but long computations shows that in the remaining cases, when $\tau_4$ avoids the other three patterns and the set $\{\tau_1,\tau_2,\tau_3,\tau_4\}$  does not contains at the same time $\id$ and $\psi,$ the generated group is the whole $S_n.$
% sostituita con:

\item Consider $\langle S_n(123, 1 3 2,213, 4 3 1 2)\rangle,$ with $n\geq 4.$
 The generating set contains the permutations $\psi_n,$ $\alpha,$ $\beta$ and $\gamma,$ where
 $$\alpha=n-1\quad n\quad n-2\quad n-3\;\ldots 1, $$
 $$\beta= n-1\quad n\quad n-3\quad n-2\quad n-4\quad n-5\;\ldots 1, $$ and
 $$\gamma=  n\quad n-2\quad n-1\quad n-3\quad n-4\quad n-5\;\ldots 1 .$$

One can easily realize that the groups $\langle S_n(123, 1 3 2,213, 4 3 1 2)\rangle$ for $n\geq 8$ are all isomorphic.
Using GAP \cite{GAP4} we verified that $\langle S_8(123, 1 3 2,213, 4 3 1 2)\rangle$ is isomorphic to the group of order 1152 appearing in the statement.

\item Let $T\in S$ be a set consisting of three (distinct) patterns of length three,  $\tau_1,\tau_2$ and $\tau_3,$ and one pattern $\tau_4$ of length four. Assume that $\tau_4$ does not contain any of the other elements of $T.$ Suppose that $\langle S_n(T)\rangle\neq S_n.$

Recall that by Theorem \ref{general_generatorsSn}  for $\langle S_n(T)\rangle $ to be different from $S_n,$ $T$ must intersect each of the following sets 

$$A=\{123,231,312,321,1234,2341,4123,4231\},$$
$$B=\{123,132,213,1234,1243,1324,2134\},$$ 
$$C=\{123,213,231,1234,2134,2341\}$$
and $$D=\{123,132,231,1234,1243,2341\}.$$

\textbf{Case 1:}

If $T$ does not contain neither $\psi_3$ nor $\psi_4,$ by Lemma \ref{rcifnotpsi} and Lemma \ref{rc_inv}, we have $\langle S_n(T)\rangle=\langle S_n(T)\cup S_n(T^r)\cup S_n(T^c)\cup S_n(T^{rc})\cup S_n(T^{-1})\cup S_n((T^{r})^{-1})\cup S_n((T^{c})^{-1})\cup S_n((T^{rc})^{-1})\rangle.$

Hence, each of the sets $T^\ast,$ where $\ast$ denotes any of $1,r,c,rc,-1,$ or of their compositions, must intersect each of the sets $A,$ $B,$  $C$ and $D$ or, equivalently, $T$ must intersect each of the sets $A^\ast,$ $B^\ast,$  $C^\ast$ and $D^\ast.$

A long but trivial computation that we realized with SageMath  \cite{sagemath} shows that the sets $T$ with these characteristics are  26. Previous considerations and results allow us to reduce ourselves to the following three sets 

% $$\{1 2 3 4, 1 3 2, 2 1 3, 3 1 2\}$$    
% $$\{1 2 3 4, 2 3 1, 1 3 2, 2 1 3\}$$     
% $$\{1 2 3 4, 2 3 1, 1 3 2, 3 1 2\}$$
% $$\{1 2 3 4, 2 3 1, 2 1 3, 3 1 2\}$$   
% $$\{2 3 1, 1 2 4 3, 2 1 3, 3 1 2\}$$    
% $$\{1 3 2, 4 3 1 2, 2 1 3, 1 2 3\}$$
% $$\{2 3 1, 1 3 2, 4 3 1 2, 1 2 3\}$$
% $$\{2 3 1, 1 3 2, 4 3 1 2, 2 1 3\}$$     
% $$\{2 3 1, 1 2 3, 4 3 1 2, 2 1 3\}$$
% $$\{3 4 2 1, 1 3 2, 2 1 3, 1 2 3\}$$ 
% $$\{3 4 2 1, 1 3 2, 3 1 2, 1 2 3\}$$  
% $$\{3 4 2 1, 1 3 2, 2 1 3, 3 1 2\}$$    
% $$\{3 4 2 1, 1 2 3, 2 1 3, 3 1 2\}$$   
% $$\{2 3 1, 1 2 3, 2 1 4 3, 3 1 2\}$$     
% $$\{2 3 1, 1 3 2, 4 1 2 3, 2 1 3\}$$      
% $$\{2 3 1, 2 1 3 4, 1 3 2, 3 1 2\}$$     
% $$\{1 3 2, 3 2 1 4, 3 1 2, 1 2 3\}$$      
% $$\{2 3 1, 1 3 2, 3 2 1 4, 1 2 3\}$$      
% $$\{2 3 1, 1 3 2, 3 2 1 4, 3 1 2\}$$    
% $$\{2 3 1, 1 2 3, 3 2 1 4, 3 1 2\}$$    
% $$\{2 3 4 1, 1 3 2, 2 1 3, 3 1 2\}$$  
% $$\{1 3 2, 2 1 3, 4 2 3 1, 1 2 3\}$$    
% $$\{1 4 3 2, 1 2 3, 2 1 3, 3 1 2\}$$    
% $$\{2 3 1, 1 4 3 2, 1 2 3, 2 1 3\}$$   
% $$\{2 3 1, 1 4 3 2, 1 2 3, 3 1 2\}$$    
% $$\{2 3 1, 1 4 3 2, 2 1 3, 3 1 2\}$$     
% From this list of 26 sets we can cancel all the sets that we already considered in the previous items and that can be obtained from these by the reverse-complement , the inverse map or their composition. 
% In this way we get the following three sets

$$\{1 3 2, 2 3 1, 3 1 2, 1 2 3 4\},\;\{1 2 3, 1 3 2, 2 3 1, 4 3 1 2\},\;\{1 2 3, 2 1 3, 2 3 1, 4 3 1 2\}.$$
% $$\{1 2 3 4, 2 3 1, 2 1 3, 3 1 2\}$$   
% $$\{3 4 2 1, 1 3 2, 3 1 2, 1 2 3\}$$ 
% $$\{3 4 2 1, 1 2 3, 2 1 3, 3 1 2\}$$   

% Notice that three of the preceding sets can be obtained from the others applying the reverse-complement or the inverse map. So we are reduced to consider three cases.

\begin{itemize}
    \item Let $G=\langle S_n(1 3 2, 2 3 1, 3 1 2, 1 2 3 4)\rangle.$
    Notice that the generating set contains only $$\alpha=n-2\quad n-3\;\ldots 1\;n-1\quad n,$$
    $$\beta=n-1\quad n-2\;\ldots 1\;n,$$ and $\psi_n.$

    Hence in $G$ we have the elements $$(\beta^{r})^{-1}=2\;3\;\ldots n\;1=(1,\;2,\;\ldots \;,n ),$$ and 

    $$\psi \beta^{-1} \psi \beta \alpha  \psi=(1,2),$$ which implies $G=S_n$ by Lemma \ref{gen_sets}.

    \item Let $G=\langle S_n(1 2 3, 2 1 3, 2 3 1, 4 3 1 2)\rangle.$
    Notice that the generating set contains only
    $$\alpha=n\quad n-2\quad n-3\;\ldots 1\;n-1,$$
    $$\beta=n-1\quad n-2\quad n-3\;\ldots 1\;n,$$ and $\psi_n.$
    The generated group thus contains $$(\alpha  \beta)^{rc}=(1,\; 2)$$ and $\beta^c=(1,\;2,\;\ldots \;,n),$ which generate $S_n$ by Lemma \ref{gen_sets}.

    \item Let $G=\langle S_n(123, 132, 2 3 1,  4 3 1 2)\rangle.$
    Notice that the generating set contains only
    $$\alpha=n\;1\;n-2\quad n-3\;\ldots 2,$$
    $$\beta=1\;n\;n-2\quad n-3\;\ldots 2$$ and $\psi_n.$
    The generated group thus contains $$(\beta  \alpha)^{rc}=(1,\; 2)$$ and $\beta^r=(1,\;2,\;\ldots \;,n)$ which generate $S_n$ by Lemma \ref{gen_sets}.
\end{itemize}

\textbf{Case 2:}

Now we consider the cases when $T$ contains $\psi_3$ or $\psi_4.$ 

If $\sigma_4=\psi_4,$ $\psi_3$ is not in $T$ since $\sigma_4$ does not contain any of $\sigma_1,$ $\sigma_2$ and $\sigma_3.$  As recalled above, if $\langle S_n(T) \rangle\neq S_n,$ $T$ must intersect each of the sets $A,B,C$ and $D.$
We can also assume that $T$ does not contain $\id_3,$ otherwise the generated group would be trivial for large $n.$

The only sets $T$ with all these properties are 

$$\{132, 213, 312, 4321\},\;\{132, 213, 231, 4321\},$$
$$\{132, 231, 312,  4321\},\;\{213, 231, 312,  4321\}.$$

The third one of these sets has been already considered in the previous items of the Theorem and the last one is the reverse-complement of the third one. 
The second case is the inverse of the first one. So we can consider only the first case.

Let $G=\langle S_n(132, 213, 312, 4321)\rangle.$
Notice that the generating set contains only
$$\alpha=(1,\;2,\;\ldots \;,n),$$
$$\beta=3\;4\;\ldots n\;2\;1,$$ and $\id_n.$

The generated group thus contains $$\alpha^2  \beta^{-1}=(1,\; 2),$$ and coincides with $S_n$ by Lemma \ref{gen_sets}.

Finally, we consider the case when $\psi_3\in T.$
In this case we can suppose that
\begin{itemize}
    \item $\sigma_4\neq \psi_4,\id_4,$
    \item $T$ does not contain $\id_3,$
    \item $\sigma_4$ does not contain any other $\sigma_i'$s,
    \item $T$ intersects the sets $A,$ $B,$ $C,$ and $D$ defined above,
    \item $T$ intersects the sets $A^{-1},B^{-1},C^{-1}$ and $D^{-1}$ (since $\langle S_n(T)\rangle=\langle S_n(T^{-1}) \rangle).$
\end{itemize}

 If these hypothesis are not satisfied the generated group is either $\{\id_n\}$ or $S_n.$ 

There are 14 of such sets $T,$ but 13 of them have been already considered.

The remaining case is 
%\{321, 1243, 213, 312\}
$$\{213, 312, 321, 1243\} .$$

Hence let $G=\langle S_n(213, 312, 321, 1243)\rangle.$
Notice that the generating set contains only
$$\alpha=2\;3\;4\;\ldots \;n\;1,$$
$$\beta=1\;3\;4\;\ldots\; n\;2$$ and $\id_n.$
The generated group thus contains $$\beta  \alpha^{-1}=(1,\; 2)$$ thus is $S_n$ by Lemma \ref{gen_sets}. 

This concludes the proof.

\end{enumerate}

\endproof

Now we consider the groups generated by sets $S_n(T)$ where $T\subseteq S_4$ and $|T|=4.$

\begin{thm}
Let $T\subseteq S_4$ with $|T|=4.$
Suppose that $T$ does not contain $\psi_4$ and that $T\neq\{1234, 1432, 3214, 4312 \}$ and $T\neq \{1234, 1432, 3214, 4231\}.$  Then $\langle S_n(T)\rangle=S_n.$
\end{thm}
\proof

Let $T\subseteq S_4\setminus\{4321\}$ with $|T|=4.$ Suppose that $S_n(T)$ generates a group different from $S_n.$

Recall from Theorem \ref{general_generatorsSn}.\textit{ii} that for $\langle S_n(T)\rangle $ to be different from $S_n,$ $T$ must intersect each of the following sets 

$$A=\{1234,2341, 4123,4231\},$$
$$B=\{1234,1243,1324,2134\},$$ 
$$C=\{1234,2134,2341\}$$
and $$D=\{1234,1243,2341\}.$$
Since $4321\notin T,$
by Lemma \ref{rcifnotpsi} and Lemma \ref{rc_inv},
$\langle S_n(T)\rangle=\langle S_n(T)\cup S_n(T^r)\cup S_n(T^c)\cup S_n(T^{rc})\cup S_n(T^{-1})\cup S_n((T^{r})^{-1})\cup S_n((T^{c})^{-1})\cup S_n((T^{rc})^{-1})\rangle.$

We are assuming that $\langle S_n(T)\rangle$ is not $S_n.$ Hence each one of the sets $T^\ast,$ where $\ast$ denotes any of $1,r,c,rc,-1$ or of their compositions, must intersect each of the sets $A,$ $B,$  $C$ and $D.$ 
% or, equivalently, $T$ must intersects each of the sets $A^\ast,$ $B^\ast,$  $C^\ast$ and $D^\ast.$

A long but trivial computation that we carried out with the aid of SageMath shows that the only sets  $T\subseteq S_4\setminus\{4321\}$ with $|T|=4$ that intersect each of the sets $A^\ast,$ $B^\ast,$ $C^\ast$ and $D^\ast$ are

$$\{1234, 3214, 3421, 4312 \},
\{1234, 1324, 3421, 4312 \},$$ $$
\{1234, 1432, 3421, 4312\},
\{1234, 1432, 3214,4312 \},$$ $$
\{1234, 1432,3214, 3421 \},
\{1234, 1432, 3214, 4231\}.$$

Notice that $$S_n(123,213,312,3421) \subseteq S_n(1234, 3214,3421, 4312 )$$ and the set on the left generates $S_n$ by Theorem  \ref{threethreeonefour}. Hence  $$\langle S_n(1234, 3214,3421, 4312 ) \rangle =S_n.$$

Similarly, $$S_n(123,132,312,3421) \subseteq S_n(1234,1324, 3421, 4312),$$
and
$$S_n(123,132, 312,3421) \subseteq S_n(1234, 1432, 3421, 4312),$$
and 
$$S_n(123,132,213,3421) \subseteq S_n(1234, 1432,3214, 3421 ),$$
hence $S_n(1234,1324, 3421, 4312 ),$ $S_n(1234, 1432, 3421, 4312)$ and $S_n(1234, 1432, 3214,3421 )$ generate $S_n.$

Using SageMath it is possible to verify that $S_n(1234, 1432,3214, 4312 )$ does not generate $S_n$ for sufficiently large $n$ $(n>20).$

In fact, for $n\geq  9,$ the last $n-9$ elements of any permutation in $S_n(1234, 1432,3214, 4312)$ are the smallest ones in decreasing order and the set $S_n(1234, 1432,3214, 4312)$ has always cardinality $145.$ To explain this fact, notice that given $\pi\in S_n(1234, 1432,3214, 4312),$ with $n\geq 9,$
\begin{itemize}
\item the position of the element $1$ is greater than or equal to $n-4,$ otherwise $\pi$ would contain $1234$ or $1432.$
\item Similarly the element $2$ has position greater than or equal to $n-5.$
\item The element $2$ precedes the element $1.$ In fact,
if the element $2$  follows the element $1,$ all the elements before $1$ must be in increasing order because $\pi$ avoids $4312.$ But this is impossible since $\pi$ avoids also $1234.$ 
\item  For the same reason $3$ precedes $2.$ Hence we have that the elements $321$ form a decreasing sequence.
\item The element $1$ is in last position, otherwise we would have an occurrence of $3214.$
\end{itemize}

Iterating this process we  conclude that the last $n-9$ elements of $\pi$ are the smallest ones in decreasing order. 
Notice that the same conclusion does not hold for the first nine elements of the permutation.
As an example consider the following permutation in $S_{21}(1234, 1432, 3214,4312 ),$
$$19\, 20\, 21\, 14\, 13\, 16\, 15\, 18\, 17\, 12\, 11\, 10\, 9\, 8\, 7\, 6\, 5\, 4\, 3\, 2\, 1. $$
As a consequence, if $n\geq 9,$ there is a bijection between $S_n(1234, 1432,3214, 4312 )$ and $S_{n+1}(1234, 1432,3214, 4312 ).$  

% Given a permutation $\pi$ in the first set, scale every element adding one and attach $1$ at the end. This produce a permutation in the second set and every permutation of the second set can be obtained in this way.

For $n=2k+1,$ $k\geq 10,$ the element $k+1$ turns out to be fixed in any permutation in this set and hence also in the generated group. Similarly for $n=2k,$ $k\geq 10,$ the elements $k$ and $k+1$ are in position $k+1$ and $k$ respectively. Hence every element in the generated group fixes the set $\{k,k+1\}.$

It can be shown similarly that neither $S_n(1234, 1432, 3214, 4231)$  generates $S_n$ for sufficiently large $n.$ In fact, it possible to prove that, given $\pi\in S_n(1234, 1432, 3214, 4231) $ then, for $n\geq 13,$ $\pi_7\pi_8\ldots \pi_{n-6}=n-6\quad n-7\, \ldots 7.$

% \todo{si direbbe che valga una cosa analoga ma qui la cardinalita \'e 289. Ad esempio per n=13 tutte hanno il 7 fisso quindi non possono generare Sn, ma come dimostrarlo? per n=14 hanno 8 7 in posizioni 7 8, per n=18 hanno tutte 12,11,10,9,8,7 in posizioni 7,8,9,10,11,12. sembra che gli elementi in posizioni 7,...,n-6 siano 7,...,n-6 in ordine decrescente.}

\endproof

\section{Open problems}

We conclude the paper with some open problems. 

Firstly, suppose that  the sequence $(S_n(T))_{n\geq 0}$ is eventually constant. What kind of groups arise in this way? For example, is it possible to obtain the dihedral group $D_k,$ for $k>4?$

Another problem is that of completing the classification of $\langle S_n(T)\rangle$ when $T$ consists of four or more patterns. In particular, we wonder whether there exists a general procedure to determine the group $\langle S_n(T)\rangle$ for any finite set of patterns $T.$

\addcontentsline{toc}{section}{Bibliography}
\bibliographystyle{plain}
\bibliography{BIBLIOGRAFIA}

\end{document}